\documentclass[reqno]{amsart}
\usepackage[T1]{fontenc}
\usepackage[margin=1in]{geometry}
\usepackage{amsmath}
\usepackage{amssymb}
\usepackage{amsthm}
\begin{document}
%%%%%%%%%%%%%%%%%%%%%%%%%%%%%%%%%%%%%%%%%%%%%%%%%%%%%%%%%%%%%%%%%%%%%%%
%%%%%%%%%%%%%%%%%%%%%%%%%     Macros      %%%%%%%%%%%%%%%%%%%%%%%%%%%%%
%%%%%%%%%%%%%%%%%%%%%%%%%%%%%%%%%%%%%%%%%%%%%%%%%%%%%%%%%%%%%%%%%%%%%%%
\def\eq#1{{\rm(\ref{#1})}}
\theoremstyle{plain}% default
\newtheorem*{theo}{Theorem}
\newtheorem*{ack}{Acknowledgements}
\newtheorem*{pro}{Proposition}
\newtheorem*{coro}{Corollary}
\newtheorem*{lemm}{Lemma}
\newtheorem{thm}{Theorem}[section]
\newtheorem{lem}[thm]{Lemma}
\newtheorem{prop}[thm]{Proposition}
\newtheorem{cor}[thm]{Corollary}
\theoremstyle{definition}
\newtheorem{dfn}[thm]{Definition}
\newtheorem*{rem}{Remark}
\def\coker{\mathop{\rm coker}}
\def\codim{\mathop{\rm codim}}
\def\ind{\mathop{\rm ind}}
\def\Re{\mathop{\rm Re}}
\def\Vol{\rm Vol}
\def\Im{\mathop{\rm Im}}
\def\im{\mathop{\rm im}}
\def\sp{\mathop{\rm span}}
\def\Hol{{\textstyle\mathop{\rm Hol}}}
\def\C{{\mathbin{\mathbb C}}}
\def\R{{\mathbin{\mathbb R}}}
\def\N{{\mathbin{\mathbb N}}}
\def\Z{{\mathbin{\mathbb Z}}}
\def\O{{\mathbin{\mathbb O}}}
\def\L{{\mathbin{\mathcal L}}}
\def\X{{\mathbin{\mathcal X}}}
\def\al{\alpha}
\def\be{\beta}
\def\ga{\gamma}
\def\de{\delta}
\def\ep{\epsilon}
\def\io{\iota}
\def\ka{\kappa}
\def\la{\lambda}
\def\ze{\zeta}
\def\th{\theta}
\def\vt{\vartheta}
\def\vp{\varphi}
\def\si{\sigma}
\def\up{\upsilon}
\def\om{\omega}
\def\De{\Delta}
\def\Ga{\Gamma}
\def\Th{\Theta}
\def\La{\Lambda}
\def\Om{\Omega}
\def\Up{\Upsilon}
\def\sm{\setminus}
\def\na{\nabla}
\def\pd{\partial}
\def\op{\oplus}
\def\ot{\otimes}
\def\bigop{\bigoplus}
\def\iy{\infty}
\def\ra{\rightarrow}
\def\longra{\longrightarrow}
\def\dashra{\dashrightarrow}
\def\t{\times}
\def\w{\wedge}
\def\bigw{\bigwedge}
\def\d{{\rm d}}
\def\bs{\boldsymbol}
\def\ci{\circ}
\def\ti{\tilde}
\def\ov{\overline}
\def\sv{\star\vp}
%%%%%%%%%%%%%%%%%%%%%%%%%%%%%%%%%%%%%%%%%%%%%%%%%%%%%%%%%%%%%%%%%%%%%%%
%%%%%%%%%%%%%%%%%%%%%     Text of paper    %%%%%%%%%%%%%%%%%%%%%%%%%%%%
%%%%%%%%%%%%%%%%%%%%%%%%%%%%%%%%%%%%%%%%%%%%%%%%%%%%%%%%%%%%%%%%%%%%%%%
\title{Subspaces of Multisymplectic Vector Spaces}

\author{A. J. Todd}

\address {Department of Mathematics, University of California - Riverside, Riverside, CA, 92521}
\email{ajtodd@math.ucr.edu}

\begin{abstract}
A notion of orthogonality in multisymplectic geometry has been developed by \cite{CIdL2} and used by many authors. In this paper, we review this concept and propose a new type of orthogonality in multisymplectic geometry; we prove a number of results regarding this orthogonality and its associated subspaces. We end by calculating the various subspaces of a $G_2$-vector space $(V^7,\vp)$ based on both types of orthogonality.
\end{abstract}

\date{}
\maketitle
%%%%%%%%%%%%%%%%%%%%%%%%%%%%%%%%%%%%%%%%%%%%%%%%%%%%%%%%%%%%%%%%%%%%%
%%%%%%%%%%%%%%%%%%%%%%%%%%%%%%%%%%%%%%%%%%%%%%%%%%%%%%%%%%%%%%%%%%%%%
\section{Introduction}
Let $V$ be an $n$-dimensional vector space, and let $\om$ be an exterior $(k+1)$-form on $V$ satisfying the nondegeneracy condition 
\begin{equation}
v\lrcorner\om=0\text{ iff  }v=0
\label{vnd}
\end{equation}
for $v\in V$. The pair $(V,\om)$ is called a \emph{multisymplectic vector space of degree $k+1$}. Such vector spaces are the tangent spaces to a \emph{multisymplectic manifold} $(M,\om)$ which is an $n$-dimensional smooth manifold $M$ together with a closed differential $(k+1)$-form $\om$ on $M$ satisfying the nondegeneracy condition 
\begin{equation}
X\lrcorner\om=0\text{ iff  }X=0
\label{vfnd}
\end{equation}
for $X$ a vector field on $M$. The pair $(M,\om)$ can be viewed as a natural generalization of symplectic manifolds which, in this language, are even-dimensional multisymplectic manifolds of degree $2$. Exact multisymplectic manifolds, that is, multisymplectic manifolds where the multisymplectic $(k+1)$-form is exact, arise naturally in physics as bundles of higher-degree differential forms equipped with an exact multisymplectic form, called multiphase spaces which are themselves generalizations of the standard phase space given by the cotangent bundle equipped with the canonical symplectic form, e. g. \cite{CIdL1, CIdL2, FPR2, FPR1, FoRo, GoIsMa}; indeed much of the interest in the subject of multisymplectic geometry has come from various areas of physics, e. g. \cite{AtWi, BHR, BR, CCI, EEMLRR1, EEMLRR2, GoVa, GYZ, He, HeKo1, HeKo2, HeKo3, PaRo, Vey}.

For a multisymplectic vector space of degree $k+1>2$, there exist stronger nondegeneracy conditions than the one given in Equation \eq{vnd}. An exterior form $\om$ of degree $k+1$ on a vector space $V$ is said to be \emph{$r$-nondegenerate} if $\om$ satisfies
\begin{equation}
 (v_1\w\cdots\w v_r)\lrcorner\om=0\text{ if and if only }v_1\w\cdots\w v_r=0
\end{equation}
for all $v_1\w\cdots\w v_r\in G^rV$ where $G^rV=\{v_1\w\cdots\w v_r:v_i\in V\}$ is the set of \emph{decomposable $r$-multivectors}, that is, $r$-fold wedge products of vectors in $V$. Madsen and Swann \cite{MaSw1} consider the cases for $r=1$ and $r=k$ and are what they refer to as \emph{weakly nondegenerate} and \emph{fully nondegenerate} respectively. When we wish to emphasize the $r$-nondegeneracy of the multisymplectic form, we will call $(V,\om)$ \emph{$r$-multisymplectic}. It will be understood that ``multisymplectic'' without further qualification will always refer to ``$1$-multisymplectic''.

\cite[Theorem $2.2$]{MaSw1} states that a vector space of dimension $n$ admits a fully nondegenerate form degree $k+1$ if and only if $k+1=n$ or the pair $(k+1,n)$ is either $(3,7)$ or $(4,8)$. The case where $k+1=n$ corresponds to $\om$ being a volume form; the case where $(k+1,n)=(3,7)$ corresponds to a \emph{$G_2$-vector space $(V,\vp)$} which is a $7$-dimensional vector space $V$ with an exterior $3$-form $\vp$, a two-fold vector cross product $\t$ and an inner product $\langle\cdot,\cdot\rangle_{\vp}$ satisfying the condition
\begin{equation}
 (u\w v\w w)\lrcorner\vp=\langle u\t v, w\rangle_{\vp}
\end{equation}
This is a $7$-dimensional, $2$-multisymplectic vector space of degree $3$. Such vector spaces arise as the tangent spaces to \emph{$7$-manifolds with $G_2$-structure}. Let $M$ be a $7$-dimensional manifold admitting a smooth differential $3$-form $\vp$ such that, for all $p\in M$, the pair $(T_pM,\vp)$ is isomorphic as an oriented vector space to the pair $(\R^7,\vp_0)$ where 
\begin{equation}
\vp_0=\d x^{123}+\d x^{145}+\d x^{167}+\d x^{246}-\d x^{257}-\d x^{347}-\d x^{356}
\label{G23form}
\end{equation}
with $\d x^{ijk}=\d x^i\w \d x^j\w \d x^k$. In \cite{Br1}, it is shown that the Lie group $G_2$ can be defined as the set of all elements of $GL(7,\R)$ that preserve $\vp_0$, so for a manifold admitting such a $3$-form, there is a reduction in the structure group of the tangent bundle to the exceptional Lie group $G_2$; hence, the pair $(M,\vp)$ is called a manifold with $G_2$-structure. Using the theory of $G$-structures and the inclusion of $G_2$ in $SO(7)$, all manifolds with $G_2$-structure are necessarily orientable and spin, any orientable $7$-manifold with spin structure admits a $G_2$-structure, and associated to a given $G_2$-structure $\vp$ are a metric $g_{\vp}$ called the \emph{$G_2$-metric}, satisfying
\begin{equation}
(X\lrcorner\vp)\w(Y\lrcorner\vp)\w\vp=6g_{\vp}(X,Y)\Vol_{\vp}
\label{G2metric}
\end{equation}
for any vector fields $X$ and $Y$ on $M$ and a $2$-fold vector cross product $\t$ satisfying
\begin{equation}
 (X\w Y\w Z)\lrcorner\vp=g_{\vp}(X\t Y,Z)
\end{equation}
for any vector fields $X$, $Y$ and $Z$ on $M$. When $\d\vp=0$, we say that $(M,\vp)$ is a manifold with \emph{closed $G_2$-structure}, and we have a $7$-dimensional, $2$-multisymplectic manifold of degree $3$. See \cite{Jo1, Jo2, Sa} for more information on these constructions and conditions.

Fundamental to the study of symplectic geometry is the notion of orthogonal complement with respect to the symplectic form as it is on this notion that the definitions of isotropic, coisotropic and Lagrangian subspaces are based. Extending this in natural way, Cantrijn, Ibort and de Le\'on \cite{CIdL1, CIdL2} define what we will call the \emph{Type-I $l^{th}$-orthogonal complement} of a linear subspace $W$ of the multisymplectic vector space $(V,\om)$ by
\begin{equation}
W_I^{\perp,l}=\left\{v\in V:\left(v\w w_1\w\cdots\w w_l\right)\lrcorner\om=0\text{ for all }w_1,\ldots,w_l\in W\right\}.
\label{WIperp}
\end{equation}
for $1\leq l\leq k$. When $k=1$, then we recover the usual definition of symplectic orthogonal complement. Using this notion of orthogonality, a subspace $W$ of a multisymplectic vector space $(V,\om)$ will be called \emph{Type-I $l$-isotropic} if $W\subseteq W_I^{\perp,l}$, \emph{Type-I $l$-coisotropic} if $W_I^{\perp,l}\subseteq W$, \emph{Type-I $l$-Lagrangian} if $W=W_I^{\perp,l}$. Unraveling the definitions shows that
\begin{equation}
 W\cap W_I^{\perp,k}=\ker(\om|_W)=\{w\in W:w\lrcorner(\om|_W)=0\}
\label{wcapwperpk}
\end{equation}
Hence, a subspace $W$ is \emph{multisymplectic} if $W\cap W_I^{\perp,k}=\{0\}$. This notion of orthogonality has been used in several articles, cf. e. g. \cite{CGM, EEdLMLRR, Gome, Rog, Sha}. 

The main purpose of this paper then is to consider the following new notion of orthogonality in a multisymplectic vector space $(V,\om)$ and to calculate the various Type-I and Type-II subspaces for a $G_2$-vector space $(V,\vp)$.
\begin{dfn}
For any subspace $W$, we define the \emph{Type-II $l^{th}$-orthogonal complement} of $W$ (with respect to $\om$) by
\begin{equation}
 W_{II}^{\perp,l}=\{v_1\w\cdots\w v_l\in G^lV:(v_1\w\cdots\w v_l \w w)\lrcorner\om=0\text{ for all }w\in W\}
 \label{typeiidefn}
\end{equation}
for any $1\leq l\leq n$.
\end{dfn}

The rest of the paper is organized as follows. In Section \ref{TypeI}, we review the results, with proof, of \cite{CIdL1} on Type-I orthogonality and the associated subspaces which we then adapt in Section \ref{TypeII} to Type-II orthogonality and the associated subspaces to get the following results.

\begin{pro}
Let $(V,\om)$ be a multisymplectic vector space of degree $k+1$, and let $U$, $W$ be any subspaces of $V$. Then for any $1\leq l\leq n$
\begin{equation}
 \{0\}_{II}^{\perp,l}=G^lV
\end{equation}
\begin{equation}
\om\text{ is $r$-nondegenerate if and only if }V_{II}^{\perp,r}=\{0\}.
\end{equation}
\begin{equation}
\text{If }\om|_W\text{ is $r$-nondegenerate, then }G^rW\cap W_{II}^{\perp,r}=\{0\}
\end{equation}
\begin{equation}
\om|_W\text{ is fully nondegenerate if and only if }G^kW\cap W_{II}^{\perp,k}=\{0\}
\end{equation}
\begin{equation}
W_{II}^{\perp,l}=G^lV\text{ for }l\geq k+1 
\end{equation}
\begin{equation}
\text{If }U\subseteq W\text{ then }W_{II}^{\perp,l}\subseteq U_{II}^{\perp,l}
\end{equation}
\begin{equation}
 (U+W)_{II}^{\perp,l} = U_{II}^{\perp,l}\cap W_{II}^{\perp, l}
\end{equation}
\end{pro}

\begin{pro}
 Let $W$ be a subspace of a multisymplectic vector space $(V,\om)$.
 \begin{enumerate}
  \item If $W$ is Type-II $l$-isotropic, then $W$ is Type-II $l'$-isotropic for all $l'\geq l$.
  \item If $W$ is Type-II $l$-coisotropic, then $W$ is Type-II $l''$-coisotropic for all $l'' \leq l$.
  \item $W$ is Type-II $l$-isotropic for all $l\geq\dim{W}$.
  \item Every subspace containing a Type-II $l$-coisotropic subspace is Type-II $l$-coisotropic.
  \item Every subspace that is contained in a Type-II $l$-isotropic subspace is Type-II $l$-isotropic.
 \end{enumerate}
\end{pro}

\begin{pro}
 Let $(V,\om)$ be an $n$-dimensional multisymplectic vector space of degree $k+1$. Then
 \begin{enumerate}
  \item Each subspace of dimension $l$ is $l$-isotropic.
  \item If $\om$ is $r$-nondegenerate, and if $W$ is a $r$-isotropic subspace of $V$, then $\codim{W}\geq k+1-r$ (and hence, $\codim(W)\geq k$)
  \item If $k+1=n$ (that is, $\om$ is a volume form for $V$) then every subspace $W$ of $V$ is Type-II $l$-Lagrangian, with $l=\dim{W}$ and, moreover, $W_{II}^{\perp,l'}=\{0\}$ for all $l'<l$.
 \end{enumerate}
\end{pro}

Section \ref{Theorem1} is devoted entirely to proving the following theorem 
\begin{thm}
Let $(V,\vp)$ be a $G_2$-vector space, and let $W$ be a subspace of $V$.
\begin{enumerate}
    \item For $\dim W=1$, $W$ is $1$-Lagrangian and both Type-I and Type-II $2$-isotropic.
    \item For $\dim W\geq2$, $W$ is $1$-coisotropic.
    \item For $\dim W=2$, $W$ is Type-I and Type-II $2$-isotropic.
    \item For $\dim W=3$,
    \begin{itemize}
        \item if $(W,\t)$ is isomorphic to $(\Im\mathbb{H},\t_0)$, $W$ is $2$-multisymplectic.
        \item otherwise, $W$ is Type-I and Type-II $2$-isotropic.
    \end{itemize}
    \item For $\dim W=4$,
    \begin{itemize}
        \item if $W$ contains a $3$-dimensional subspace $(\tilde{W},\t)$ isomorphic to $(\Im\mathbb{H},\t_0)$, $W$ is Type-I $2$-coisotropic but is none of Type-II $2$-isotropic, Type-II $2$-coisotropic nor $2$-multisymplectic.
        \item otherwise, $W$ is Type-I $2$-Lagrangian and Type-II $2$-isotropic.
    \end{itemize}
    \item For $\dim W=5$, $W$ is Type-I $2$-coisotropic and $1$-multisymplectic, but is none of Type-II $2$-isotropic, Type-II $2$-coisotropic nor Type-II $2$-multisymplectic.
    \item For $\dim W=6$, $W$ is Type-I and Type-II $2$-coisotropic and $1$-multisymplectic.
\end{enumerate}
\end{thm}

This allows us to give the following characterizations of the associative and coassociative submanifolds of a $G_2$-manifold.

\begin{coro}
Let $(V,\vp)$ be a $G_2$-vector space. Then the associative subspaces are the $3$-dimensional subspaces which are $2$-multisymplectic and the coassociative subspaces are the $4$-dimensional subspaces which are Type-I $2$-Lagrangian/Type-II $2$-isotropic.
\end{coro}

\begin{coro}
Let $(M,\vp)$ be a $G_2$-manifold. Then the associative submanifolds are exactly those $3$-dimensional submanifolds which are $2$-multisymplectic, and the coassociative submanifolds are exactly those $4$-dimensional submanifolds which are Type-I $2$-Lagrangian/Type-II $2$-isotropic.
\end{coro}

\section{Type-I Orthogonality}
\label{TypeI}

Note that there is the filtration of orthogonal complements
\begin{equation}
W_I^{\perp,1}\subseteq W_I^{\perp,2}\subseteq\cdots\subseteq W_I^{\perp,k},
\label{filtration}
\end{equation}
and if $l>\dim W$, then $W_I^{\perp,l}=V$. 

\begin{prop}[\cite{CIdL2}, \emph{Proposition $3.1$}]
Let $(V,\om)$ be a multisymplectic vector space of degree $k+1$, and let $U$, $W$ be any subspaces of $V$. Then for any $l,l_1,l_2\in\{1,\ldots,k\}$
\begin{equation}
 \{0\}_I^{\perp,l}=V
\label{cidl2311}
\end{equation}
\begin{equation}
 V_I^{\perp,l}=\{0\}
\label{cidl2312}
\end{equation}
\begin{equation}
\text{If }U\subseteq W\text{ then }W_I^{\perp,l}\subseteq U_I^{\perp,l}
\label{cidl2313}
\end{equation}
\begin{equation}
 (U+W)_I^{\perp,l}\subseteq U_I^{\perp,l}\cap W_I^{\perp, l}
\label{cidl2314}
\end{equation}
\begin{equation}
 U_I^{\perp,l_1}\cap W_I^{\perp,l_2}\subseteq (U+W)_I^{\perp,l_1+l_2-1}\text{ for }l_1+l_2\leq k+1
\label{cidl2315}
\end{equation}
\begin{equation}
U_I^{\perp,l_1}+W_I^{\perp,l_2}\subseteq (U\cap W)_I^{\perp,\ov{l}}\text{ with }\ov{l}=\max\{l_1,l_2\}
\label{cidl2316}
\end{equation}
\end{prop}

\begin{proof}[Sketch of Proof]
Equations \eq{cidl2311}, \eq{cidl2312}, \eq{cidl2313} follow immediately from the definition. Equation \eq{cidl2314} is then a straightforward consequence of Equation \eq{cidl2313}. To see Equation \eq{cidl2315}, let $v\in U_I^{\perp,l_1}\cap W_I^{\perp,l_2}$ and $u_1+w_1,\ldots,u_{l_1+l_2-1}+w_{l_1+l_2-1}\in U+W$. Then expanding
\begin{equation}
 (v\w(u_1+w_1)\w\cdots\w(u_{l_1+l_2-1}+w_{l_1+l_2-1})\lrcorner\om
\end{equation}
gives a sum of terms each of which has $p_i$ entries from $U$ and $q_i$ entries from $W$ where $p_i+q_i=l_1+l_2-1$. For a given term, if $p_i<l_1$, then $q_2\geq l_2$ so that term vanishes because $v\in W_I^{\perp,l_2}$; on the other hand, if for a given term $p_i\geq l_1$, then that term vanishes because $v\in U_I^{\perp,l_1}$ giving the relation. Finally, Equations \eq{cidl2313}, \eq{filtration} together with closure imply Equation \eq{cidl2316}.
\end{proof}

\begin{cor}[\cite{CIdL2}, \emph{Corollary $3.2$}]
\begin{equation}
 (U+W)_I^{\perp,1}=U_I^{\perp,1}\cap W_I^{\perp,1}
\end{equation}
\label{cidl232}
\end{cor}

These definitions together with Equations \eq{filtration}, \eq{cidl2313} imply a number of properties. If a subspace $W$ is Type-I $l$-isotropic, then it is $l'$-isotropic for all $l'\geq l$. If $W$ is $l$-coisotropic, then it is $l''$-coisotropic for all $l''\leq l$. $W$ is $l$-isotropic for all $l\geq\dim{W}$. Every subspace containing an $l$-coisotropic subspace is $l$-coisotropic. And every subspace of an $l$-isotropic subspace is $l$-isotropic.

\begin{prop}[\cite{CIdL2}, \emph{Proposition $3.4$}]
 Let $(V,\om)$ be an $n$-dimensional multisymplectic vector space of degree $k+1$. Then
 \begin{enumerate}
  \item Each subspace of dimension $1$ (resp. codimension $1$) is $1$-isotropic (resp. $k$-coisotropic)
  \item If $W$ is a $1$-isotropic subspace of $V$, then $\codim{W}\geq k$
  \item If $W$ is an $l$-isotropic subspace of $V$, then for every $l'\geq l$, there exists an $l'$-Lagrangian subspace which contains $W$.
  \item If $k+1=n$ (that is, $\om$ is a volume form for $V$) then every subspace $W$ of $V$ is $l$-Lagrangian, with $l=\dim{W}$ and, moreover, $W^{\perp,l'}=\{0\}$ for all $l'<l$.
 \end{enumerate}
\end{prop}

\begin{proof}[Proof Adapted from \cite{CIdL2}]
 \begin{enumerate}
  \item That a $1$-dimensional subspace is $1$-isotropic follows immediately from the definition of the $1$-orthogonal complement. Now consider an $(n-1)$-dimensional subspace $W$ and let $v\in W^{\perp,k}$. If $v\not\in W$, then it is easy to show a contradiction to the nondegeneracy condition on $\om$; hence $W^{\perp,k}\subseteq W$ showing that an $(n-1)$-dimensional subspace is $k$-coisotropic.
  \item Let $W$ be a $1$-isotropic subspace of $V$, and let $U$ be any complementary subspace of $V$ such that $V=W\op U$. Then for any $w\in W$, the $k$-form $w\lrcorner\om$ vanishes when contracted with any vector of $W$; hence there must exist $k$ (linearly independent) vectors in $U$ such that $(w\w u_1\w\cdots\w u_k)\lrcorner\om\neq 0$ since otherwise we get a contradiction to the nondegeneracy condition on $\om$, and thus $\dim{U}\geq k$.
  \item From above, we know that $l$-isotropic implies $l'$-isotropic for every $l'\geq l$, so it is enough to show that, for an $l$-isotropic subspace $W$ of $V$ there exists an $l$-Lagrangian subspace which contains it; therefore, assume that $W$ is $l$-isotropic and that there is some nonzero $v\in W^{\perp,l}\setminus W$. Let $W_1=W\op span\{v\}$. We now show that $W_1$ is $l$-isotropic, so consider the $l+1$ vectors $w_i+\la_iv\in W_1$. Then expanding
  \begin{equation}
    \left((w_1+\la_1v)\w\cdots\w(w_{l+1}+\la_{l+1}v)\right)\lrcorner\om
  \end{equation}
  gives a sum of terms, each of which has $l$ factors from $W$ together with $v$, and so therefore vanishes by definition of $W^{\perp,l}$. Thus $W_1$ is an $l$-isotropic subspace of $V$, and we have the inclusions
  \begin{equation}
   W\subseteq W_1\subseteq W_1^{\perp,l}\subseteq W^{\perp,l}
  \end{equation}
  where the final inclusion follows from Equation \eq{cidl2313}. Continuing in this way, we can construct an ascending chain of $l$-isotropic subspaces which necessarily possesses a maximal element which will be $l$-Lagrangian by construction.
  \item Let $W$ be an $l$-dimensional subspace of $V$. By definition of $W^{\perp,l}$, we necessarily have $W\subseteq W^{\perp,l}$. Conversely, let $v\in W^{\perp,l}$, let $\{v_1,\ldots,v_l\}$ be a basis for $W$ and $\{v_1,\ldots,v_l,v_{l+1},\ldots,v_n\}$ its completion to a basis for $V$. Writing $v=\sum_{i=1}^n\la_iv_i$,
  \begin{equation}
  \begin{split}
   0=&(v\w v_1\w\cdots\w v_l\w v_{l+2}\w\cdots\w v_n)\lrcorner \om \\
   =&(-1)^l\la_{l+1}\underbrace{(v_1\w\cdots\w v_n)\lrcorner\om}_{\neq 0} \\
  \end{split}
  \end{equation}
  implies that $\la_{l+1}=0$. Similarly, $\la_{l+2}=\cdots=\la_{n}=0$ which shows that $v\in W$ and hence that $W$ is $l$-Lagrangian. Finally, let $v\in W^{\perp,l-1}$, and again write $v=\sum_{i=1}^n\la_iv_i$. Since, for any $i=1,\ldots,n$,
  \begin{equation}
  \begin{split}
   0=&(v\w v_1\w\cdots\w \hat{w_i}\w \cdots\w v_n)\lrcorner \om \\
   =&(-1)^{i-1}\la_{i}\underbrace{(v_1\w\cdots\w v_n)\lrcorner\om}_{\neq 0} \\
  \end{split}
  \end{equation}
  implies that $\la_i=0$, this shows that $v=0$ and hence that $W^{\perp,l-1}=\{0\}$. The complete result then follows using the filtration given by Equation \eq{filtration}.
 \end{enumerate}
\end{proof}

\begin{rem}
 In fact, one can show a little more here. If $W$ is an $l$-dimensional subspace of $V$, then $W$ is $l$-isotropic.
\end{rem}

Note that since $\{0\}$ is $l$-isotropic for any $l$, the above proposition implies that existence of $l$-Lagrangian subspaces. Further, by the proof, $l$-Lagrangian subspaces are the maximal elements of the nonempty partially ordered set of $l$-isotropic subspaces with respect to inclusion, so we also have existence in the infinite dimensional case. Unlike the symplectic case, $l$-Lagrangian subspaces need not all have the same dimension. 

%We last remark that we get nothing more in this case if we consider general multivectors rather than simply decomposable multivectors. That is, if we define
%\begin{equation}
% \tilde{W}^{\perp,l}=\{v\in V:(v\w Q)\lrcorner\om=0\text{ for all }Q\in\La^lW\}
%\end{equation}
%then we necessarily have
%\begin{equation}
% \tilde{W}^{\perp,l}=W^{\perp,l}
%\end{equation}

\section{Type-II Orthogonality}
\label{TypeII}
In this section, we consider Type-II orthogonality defined in Equation \eq{typeiidefn}. 

\begin{rem}
 Since $G^1V=V$, $W_I^{\perp,1}=W_{II}^{\perp,1}$ where $W_I^{\perp, 1}$ is the Type-I $1^{st}$-orthogonal complement.
\end{rem}

For $v_1\w\cdots\w v_{l_1}\in W_{II}^{\perp,l_1}$. $\tilde{v}_1\w\cdots\w \tilde{v}_{l_2}\in W_{II}^{\perp,l_2}$, $v_1\w\cdots\w v_{l_1}\w\tilde{v}_1\w\cdots\w \tilde{v}_{l_2}\in W_{II}^{\perp,l_1+l_2}$. That is, we have closure under the wedge product, so that the collection of these spaces
\begin{equation}
 W_{II}^{\perp}=\bigoplus_{l=1}^nW_{II}^{\perp,l}
\end{equation}
forms a semigroup with respect to $\w$. In fact, more is true here. Let $v_1\w\cdots\w v_{l_1}\in W_{II}^{\perp,l_1}$ and let $\tilde{v}_1\w\cdots\w \tilde{v}_{l_2}\in G^{l_2}V$. Then $v_1\w\cdots\w v_{l_1}\w\tilde{v}_1\w\cdots\w\tilde{v}_{l_2}\in W_{II}^{\perp,l_1+l_2}$.

The following properties are straightforward from the definition of Type-II orthogonality.
\begin{prop}
Let $(V,\om)$ be a multisymplectic vector space of degree $k+1$, and let $U$, $W$ be any subspaces of $V$. Then for any $1\leq l\leq n$
\begin{equation}
 \{0\}_{II}^{\perp,l}=G^lV
\label{typeii0perp}
\end{equation}
\begin{equation}
\om\text{ is $r$-nondegenerate if and only if }V_{II}^{\perp,r}=\{0\}.
\label{typeiiVperp}
\end{equation}
\begin{equation}
\text{If }\om|_W\text{ is $r$-nondegenerate, then }G^rW\cap W_{II}^{\perp,r}=\{0\}
\label{typeiiWrnd}
\end{equation}
\begin{equation}
\om|_W\text{ is fully nondegenerate if and only if }G^kW\cap W_{II}^{\perp,k}=\{0\}
\label{typeiiWknd}
\end{equation}
\begin{equation}
W_{II}^{\perp,l}=G^lV\text{ for }l\geq k+1 
\end{equation}
\begin{equation}
\text{If }U\subseteq W\text{ then }W_{II}^{\perp,l}\subseteq U_{II}^{\perp,l}
\label{typeiisubseteqperp}
\end{equation}
\begin{equation}
 (U+W)_{II}^{\perp,l} = U_{II}^{\perp,l}\cap W_{II}^{\perp, l}
\label{typeiiaddperp}
\end{equation}
\end{prop}

Recall that if $\om$ is $r$-nondegenerate for some $r$, then $\om$ is $r'$-nondegenerate for all $1\leq r'\leq r$. Then Equation \eq{typeiiVperp} immediately implies that if $V^{\perp,r}=\{0\}$ for some $r$, then $V^{\perp,r'}=\{0\}$ for all $1\leq r'\leq r$; moreover, if $\om|_W$ is fully nondegenerate, then Equations \eq{typeiiWknd}, \eq{typeiiWrnd} imply $G^lW\cap W_{II}^{\perp,l}=\{0\}$ for all $1\leq l\leq k$.

\begin{dfn}
 Let $W$ be a subspace of a multisymplectic vector space $(V,\om)$. Then $W$ will be called
 \begin{enumerate}
  \item \emph{Type-II $l$-isotropic} if $G^lW\subseteq W^{\perp,l}_{II}$
  \item \emph{Type-II $l$-coisotropic} if $W^{\perp,l}_{II}\subseteq G^lW$
  \item \emph{Type-II $l$-Lagrangian} if $G^lW=W^{\perp,l}_{II}$
 \end{enumerate}
\end{dfn}

\begin{prop}
 Let $W$ be a subspace of a multisymplectic vector space $(V,\om)$.
 \begin{enumerate}
  \item If $W$ is Type-II $l$-isotropic, then $W$ is Type-II $l'$-isotropic for all $l'\geq l$.
  \item If $W$ is Type-II $l$-coisotropic, then $W$ is Type-II $l''$-coisotropic for all $l'' \leq l$.
  \item $W$ is Type-II $l$-isotropic for all $l\geq\dim{W}$.
  \item Every subspace containing a Type-II $l$-coisotropic subspace is Type-II $l$-coisotropic.
  \item Every subspace that is contained in an Type-II $l$-isotropic subspace is Type-II $l$-isotropic.
 \end{enumerate}
\end{prop}

\begin{proof}
\mbox{}
\begin{enumerate}
  \item Assume $l'\geq l$, and let $w_1\w\cdots\w w_{l'}\in G^{l'}W$. Then for any $w\in W$,
  \begin{equation}
   (w_1\w\cdots\w w_{l'}\w w)\lrcorner\om=(-1)^{l'-l}(w_{l+1}\w\cdots\w w_{l'})\lrcorner(w_1\w\cdots\w w_l\w w)\lrcorner\om=0
  \end{equation}
  showing that $w_1\w\cdots\w w_{l'}\in W_{II}^{\perp,l'}$ and hence that $W$ is Type-II $l'$-isotropic.
  
  \item Assume that $l''\leq l$ and let $v_1\w\cdots\w v_{l''}\in W_{II}^{\perp,l''}$. If $v_1\w\cdots\w v_{l''}\not\in G^{l''}W$, then there exists some $i$ such that $v_i\not\in W$. Without loss of generality, assume that $v_1\not\in W$. Let $v_{l''+1},\ldots,v_l\in V$ such that $v_1\w\cdots\w v_{l''}\w v_{l''+1}\w\cdots\w v_l\neq 0$. Then for any $w\in W$,
  \begin{equation}
   (v_1\w\cdots\w v_{l''}\w v_{l''+1}\w\cdots\w v_l\w w)\lrcorner\om=(-1)^{l-l''}(v_{l''+1}\w\cdots\w v_l)\lrcorner(v_1\w\cdots\w v_{l''}\w w)\lrcorner\om=0
  \end{equation}
  showing that $v_1\w\cdots\w v_{l''}\w v_{l''+1}\w\cdots\w v_l\in W_{II}^{\perp,l}$, but since $v_1\not\in W$, $v_1\w\cdots\w v_{l''}\w v_{l''+1}\w\cdots\w v_l\not\in G^lW$, a contradiction to our assumption that $W$ is Type-II $l$-coisotropic. Thus, we must have $v_1\w\cdots\w v_{l''}\in G^{l''}W$, showing that $W$ is Type-II $l''$-coisotropic.
  
  \item By the first part of this proposition, it is enough to consider the case where $l=\dim{W}$. Then for any $w_1\w\cdots\w w_l\in G^lW$ and any $w\in W$, we necessarily have $w_1\w\cdots\w w_l\w w=0$, so that $W$ is $l$-isotropic.
  
  \item Let $W$ be a subspace of $V$, and let $U$ be a subspace of $W$ such that $U$ is Type-II $l$-coisotropic. Then $U_{II}^{\perp,l}\subseteq G^lU$. Since $U\subseteq W$, we have $W_{II}^{\perp,l}\subseteq U_{II}^{\perp,l}$ by Equation \eq{typeiisubseteqperp}; also, $U\subseteq W$ implies $G^lU\subseteq G^lW$. Thus
  \begin{equation}
   W_{II}^{\perp,l}\subseteq U_{II}^{\perp,l}\subseteq G^lU\subseteq G^lW
  \end{equation}
  showing that $W$ is Type-II $l$-coisotropic.
  
  \item Let $W$ be a Type-II $l$-isotropic subspace of $V$ and let $U$ be a subspace of $W$. Then $W_{II}^{\perp,l}\subseteq U_{II}^{\perp,l}$ by Equation \eq{typeiisubseteqperp}, and $G^lU\subseteq G^lW$. Finally, $W$ is Type-II $l$-isotropic, we get the following chain
  \begin{equation}
   G^lU\subseteq G^lW\subseteq W_{II}^{\perp,l}\subseteq U_{II}^{\perp,l}
  \end{equation}
  showing that $U$ is Type-II $l$-isotropic.
 \end{enumerate}
\end{proof}

\begin{prop}
 Let $(V,\om)$ be an $n$-dimensional multisymplectic vector space of degree $k+1$. Then
 \begin{enumerate}
  \item Each subspace of dimension $l$ is Type-II $l$-isotropic.
  \item If $\om$ is $r$-nondegenerate, and if $W$ is a Type-II $r$-isotropic subspace of $V$, then $\codim{W}\geq k+1-r$
  \item If $k+1=n$ (that is, $\om$ is a volume form for $V$) then every subspace $W$ of $V$ is Type-II $l$-Lagrangian, with $l=\dim{W}$ and, moreover, $W_{II}^{\perp,l'}=\{0\}$ for all $l'<l$.
 \end{enumerate}
\end{prop}

\begin{proof}
 \mbox{}
 \begin{enumerate}
  \item This is clear from the definition of $W_{II}^{\perp,l}$ and the fact that $w_1\w\cdots\w w_i=0$ for any $i>l$.
  
  \item Let $U$ be any subspace of $V$ such that $V=W\op U$. Since $W$ is Type-II $r$-isotropic, we have $G^rW\subseteq W_{II}^{\perp,r}$. Thus, for any nonzero $w_1\w\cdots\w w_r\in G^rW$, the $(k+1-r)$-form
  \begin{equation}
   (w_1\w\cdots\w w_r)\lrcorner\om
  \end{equation}
   vanishes when contracted with any vector $w\in W$. Hence for a fixed nonzero $w_1\w\cdots\w w_r\in G^rW$ there must exist $k+1-r$ linearly-independent vectors $u_1,\ldots,u_{k+1-r}\in U$, such that
   \begin{equation}
    (w_1\w\cdots\w w_r\w u_1\w\cdots\w u_{k+1-r})\lrcorner\om\neq 0
   \end{equation}
   Otherwise, we obtain a contradiction to the $r$-nondegeneracy of $\om$. Thus, $\dim{U}\geq k+1-r$.
   
   \item Let $(V,\om)$ be an $n$-dimensional multisymplectic vector space of degree $n$, so that $\om$ is a volume form for $V$; let $W$ be any $l$-dimensional subspace of $V$. Then by the first part of this proposition, $W$ is $l$-isotropic, so it suffices to show that $W$ is $l$-coisotropic. Let $v_1\w\cdots\w v_l\in W_{II}^{\perp,l}$ be nonzero. By definition of $W_{II}^{\perp,l}$, we have
   \begin{equation}
    (v_1\w\cdots\w v_l\w w)\lrcorner\om=0
   \end{equation}
   for all $w\in W$. Since $\om$ is a volume form, it is fully nondegenerate, and hence we must have $v_1\w\cdots\w v_l\w w=0$ for all $w\in W$. By definition of $\w$, this means that $w\in span\{v_1,\ldots,v_l\}$, that is, that $W\subseteq span\{v_1,\ldots,v_l\}$. Since $\dim{W}=l$, equality follows. Therefore $v_1\w\cdots\w v_l\in G^lW$, and hence $W$ is $l$-Lagrangian.
    
    For the final claim, note that
    \begin{equation}
     W_{II}^{\perp,l-1}=\{v_1\w\cdots\w v_{l-1}:(v_1\w\cdots\w v_{l-1}\w w)\lrcorner\om=0\text{ for all }w\in W\}
    \end{equation}
    Since $\om$ is fully nondegenerate, we must have $v_1\w\cdots\w v_{l-1}\w w=0$ for all $w\in W$. This implies that $v_1\w\cdots\w v_{l-1}=0$; otherwise, since $\dim{W}=l$, there would be a nonzero $w\in W$ such that $v_1\w\cdots v_{l-1}\w w\neq 0$. Thus, $W_{II}^{\perp,l-1}=\{0\}$. The general case follows by noting that if $W_{II}^{\perp,l'}\neq\{0\}$ for some $l'<l-1$, then we would have to have $W_{II}^{\perp,l-1}\neq\{0\}$.
 \end{enumerate}
\end{proof}

We end this section by remarking that for any subspace $W$, we always have $W_I^{\perp,1}=W_{II}^{\perp,1}$, so the notions of Type-I and Type-II $1$-isotropic subspaces are equivalent; similarly for Type-I and Type-II $1$-coisotropic as well as Type-I and Type-II $1$-Lagrangian. Thus, as appropriate, we will refer to simply \emph{$1$-isotropic}, \emph{$1$-Lagrangian} and \emph{$1$-coisotropic} subspaces.  Further, we see that a subspace $W$ that is multisymplectic in the sense that $W\cap W_I^{\perp,k}=\{0\}$ is, by definition, $1$-multisymplectic which is equivalent to $W\cap W_{II}^{\perp,1}=\{0\}$. 

\section{Proof of Theorem $1$}
\label{Theorem1}
Consider a linear subspace $W$ of the $G_2$ vector space $(V,\vp)$. Since $\vp$ has degree $3$, we have the Type-I and Type-II orthogonal complements $W_I^{\perp,1}$, $W_I^{\perp,2}$, $W_{II}^{\perp,1}$ and $W_{II}^{\perp,2}$; let $W^{\perp}$ denote the orthogonal complement of $W$ with respect to the metric $\langle\cdot,\cdot\rangle_{\vp}$. Note that from above, since we will always have $W_I^{\perp,1}=W_{II}^{\perp,1}$, we will denote these spaces simply by $W^{\perp,1}$.

\begin{prop}
For $\dim W=1$, $W$ is $1$-Lagrangian and Type-I and Type-II $2$-isotropic with respect to $\vp$.
\end{prop}

\begin{proof}
Fix any nonzero vector $w\in W$. Unraveling definitions, we see that
\begin{equation}
\begin{split}
W^{\perp,1}&=\left\{v\in V:\left((v\w w_1)\lrcorner\vp\right)(\tilde{v})=0\text{ for all }w_1\in W\text{, }\tilde{v}\in V\right\}\\
&=\left\{v\in V:\langle v\t w_1,\tilde{v}\rangle_{\vp}=0\text{ for all }w_1\in W\text{, }\tilde{v}\in V\right\}\\
&=\left\{v\in V:\langle v\t \al w,\tilde{v}\rangle_{\vp}=0\text{ for all }\al\in\R\text{, }\tilde{v}\in V\right\}\\
&=\left\{v\in V:\langle v\t w,\tilde{v}\rangle_{\vp}=0\text{ for all }\tilde{v}\in V\right\}\\
\end{split}
\end{equation}
which, since $W$ is $1$-dimensional, immediately implies $W\subseteq W^{\perp,1}$. Next, assume $v\in W^{\perp,1}$ is nonzero, then by the above
\begin{equation}
\langle v\t w, \tilde{v}\rangle_{\vp}=0
\end{equation}
for all $\tilde{v}\in V$ which, by the nondegeneracy of the metric implies that $v\t w=0$ which occurs if and only if $v$, $w$ are linearly dependent, i. e., $v\in W$; therefore, all $1$-dimensional subspaces of $(V,\vp)$ are $1$-Lagrangian. 

That $W$ is Type-I $2$-isotropic follows directly from the fact that $2>\dim W$ implies $W_I^{\perp,2}=V$. Further, since $G^2W=\{0\}$, we immediately have that $G^2W\subseteq W_{II}^{\perp,2}$. To see that $W$ is not Type-II $2$-Lagrangian, note
\begin{equation}
\begin{split}
 W_{II}^{\perp,2}&=\{v_1\w v_2\in G^2V:(v_1\w v_2\w w_1)\lrcorner\vp=0\text{ for all }w_1\in W\} \\
 &=\{v_1\w v_2:\al(v_1\w v_2\w w)\lrcorner\vp=0\text{ for all }\al\in\R\} \\
 &=\{v_1\w v_2:(v_1\w v_2\w w)\lrcorner\vp=0\}\\
 &=\{v_1\w v_2:\langle v_1\t v_2, w\rangle_{\vp}=0\}\\
\end{split}
\end{equation}
Hence for any nonzero $v\in V-W$, $v\w w\in W_{II}^{\perp,2}$ showing that $W_{II}^{\perp,2}\neq\{0\}$ and thus that $W$ is specifically Type-II $2$-isotropic.
\end{proof}

\begin{prop}
For $\dim W\geq2$, $W$ is $1$-coisotropic with respect to $\vp$.
\end{prop}

\begin{proof}
Let $\{e_1,e_2\}$ be any basis for $W$. Then, since $e_1\t e_2\neq 0$, it follows that $e_1,e_2\not\in W^{\perp,1}$ by nondegeneracy of the metric $\langle\cdot,\cdot\rangle_{\vp}$, so $W\cap W^{\perp,1}=\{0\}$ showing that $W^{\perp,1}\subseteq W^{\perp}$. Conversely, let $v\in W^{\perp}$, and fix a nonzero $w\in W$. Then, if $0=\langle v\t w, \tilde{v}\rangle_{\vp}$ for all $\tilde{v}\in V$, we must have $v\t w=0$ which means that $v$ and $w$ are linearly dependent, showing that $v\in W$. Since $W\cap W^{\perp}=\{0\}$, this shows that $v=0$, proving that $W^{\perp,1}=\{0\}$ when $\dim W=2$. A similar argument then shows that $W^{\perp,1}=\{0\}$ for any subspace $W$ with $\dim W\geq 2$; this implies that any subspace $W$ with $\dim W\geq 2$ is $1$-coisotropic.
\end{proof}

\begin{prop}
 For $\dim W=2$, $W$ is Type-I and Type-II $2$-isotropic with respect to $\vp$.
\end{prop}

\begin{proof}
Let $\{e_1,e_2\}$ be any basis for $W$. By definition of $W_I^{\perp, 2}$, we have
\begin{equation}
\begin{split}
W_I^{\perp, 2}&=\left\{v\in V:(v\w w_1\w w_2)\lrcorner\vp=0\text{ for all }w_1,w_2\in W\right\}\\
&=\left\{v\in V:(w_1\w w_2\w v)\lrcorner\vp=0\text{ for all }w_1,w_2\in W\right\}\\
&=\left\{v\in V:\langle w_1\t w_2, v\rangle_{\vp}=0\text{ for all }w_1,w_2\in W\right\}\\
\end{split}
\end{equation}
In the basis $\{e_1,e_2\}$ for $W$ from above, we have $w_1=\al_{11}e_1+\al_{12}e_2$ and $w_2=\al_{21}e_1+\al_{22}e_2$ for some $\al_{ij}\in\R$, $i,j=1,2$. Then we calculate that
\begin{equation}
\begin{split}
w_1\t& w_2=(\al_{11}e_1+\al_{12}e_2)\t(\al_{21}e_1+\al_{22}e_2)\\
=&\al_{11}\al_{21}\underbrace{(e_1\t e_1)}_{=0}+\al_{11}\al_{22}(e_1\t e_2)+\al_{12}\al_{21}\underbrace{(e_2\t e_1)}_{=-(e_1\t e_2)} +\al_{12}\al_{22}\underbrace{(e_2\t e_2)}_{=0}\\
=&(\al_{11}\al_{22}-\al_{12}\al_{21})(e_1\t e_2)
\end{split}
\end{equation}
Now, since we may reasonably assume that $w_1$, $w_2$ are linearly independent, it follows that $\al_{11}\al_{22}-\al_{12}\al_{21}\neq 0$. Thus, we can refine the above to
\begin{equation}
\begin{split}
W_I^{\perp, 2}&=\left\{v\in V:\langle w_1\t w_2, v\rangle_{\vp}=0\text{ for all }w_1,w_2\in W\right\}\\
&=\left\{v\in V:\langle (\al_{11}\al_{22}-\al_{12}\al_{21})(e_1\t e_2), v\rangle_{\vp}=0\text{ for all }\al_{11}\al_{22}-\al_{12}\al_{21}\neq 0\right\}\\
&=\left\{v\in V:(\al_{11}\al_{22}-\al_{12}\al_{21})\langle (e_1\t e_2), v\rangle_{\vp}=0\text{ for all }\al_{11}\al_{22}-\al_{12}\al_{21}\neq 0\right\}\\
&=\left\{v\in V:\langle e_1\t e_2, v\rangle_{\vp}=0\right\}\\
\end{split}
\end{equation}
It is then immediate from $0=\langle e_i, e_1\t e_2\rangle_{\vp}$, $i=1,2$, that $W\subseteq W_I^{\perp, 2}$. Further, it is clear from the fact that 
\begin{equation}
\dim W_I^{\perp,2}=\dim\left(\left(span\{e_1\t e_2\}\right)^{\perp}\right)=6
\end{equation}
that $W\neq W_I^{\perp, 2}$. Thus, for $\dim W=2$, $W$ is Type-I $2$-isotropic.

Next, consider
\begin{equation}
\begin{split}
 W_{II}^{\perp, 2}&=\{v_1\w v_2:(v_1\w v_2\w w)\lrcorner\vp=0\text{ for all }w\in W\}\\
 &=\{v_1\w v_2:\langle v_1\t v_2, w\rangle_{\vp}=0\text{ for all }w\in W\}
\end{split}
\end{equation}
Let $w_1\w w_2\in G^2W$ be nonzero. Since $\langle w_1\t w_2, w_i\rangle_{\vp}=0$ for $i=1,2$, we see that $w_1\w w_2\in W_{II}^{\perp,2}$ and hence $G^2W\subseteq W_{II}^{\perp,2}$. 

To see that the two are not equal in general, consider $(\R^7,\vp_0)$, and let 
\begin{equation}
 W=span\left\{\frac{\pd}{\pd x^1},\frac{\pd}{\pd x^2}\right\}
\end{equation}
In this case, a quick calculation shows that we have 
\begin{equation}
 \frac{\pd}{\pd x^1}\w\frac{\pd}{\pd x^4}\in W_{II}^{\perp,2}
\end{equation}
Thus, a $2$-dimensional subspace $W$ is Type-II $2$-isotropic.
\end{proof}

\begin{prop}
For $\dim W=3$,
\begin{itemize}
    \item if $(W,\t)$ is isomorphic to $(\Im\mathbb{H},\t_0)$, $W$ is $2$-multisymplectic.
    \item otherwise, $W$ is Type-I and Type-II $2$-isotropic.
\end{itemize}
\end{prop}

\begin{proof}
Let $\{e_1,e_2,e_3\}$ be any orthonormal basis for $W$ with respect to $\langle\cdot,\cdot\rangle_{\vp}$, and assume that $e_i\t e_j\in W$ for some $1\leq i<j\leq 3$; without loss of generality, assume that $e_1\t e_2\in W$. Then writing $e_1\t e_2=\al e_1+\be e_2+\ga e_3$ and using the orthogonality property of $\t$, we see that
\begin{equation}
0=\langle e_1\t e_2, e_1\rangle_{\vp}=\langle(\al e_1+\be e_2+\ga e_3),e_1\rangle_{\vp}=\al
\end{equation}
and
\begin{equation}
0=\langle e_1\t e_2, e_2\rangle_{\vp}=\langle(\al e_1+\be e_2+\ga e_3),e_2\rangle_{\vp}=\be.
\end{equation}
Thus, $e_1\t e_2=\ga e_3$. From here, note that
\begin{equation}
\begin{split}
\ga^2&=\langle\ga e_3,\ga e_3\rangle_{\vp}=\langle e_1\t e_2, e_1\t e_2\rangle_{\vp}=\vp(e_1,e_2,e_1\t e_2)\\
&=-\vp(e_1,e_1\t e_2,e_2)=-\langle e_1\t(e_1\t e_2), e_2\rangle_{\vp}
\end{split}
\end{equation}
Using the identity $x\t(x\t y)=-\left|x\right|^2y+\langle x,y\rangle_{\vp} x$, we have that
\begin{equation}
\ga^2=-\langle e_1\t(e_1\t e_2), e_2\rangle_{\vp}=-\langle -e_2,e_2\rangle_{\vp}=1
\end{equation}
showing that 
\begin{equation}
 e_1\t e_2=\pm e_3
 \label{e1e2e3}
\end{equation}
Using the same identity, we have that
\begin{equation}
e_1\t e_3=\pm e_1\t(e_1\t e_2)=\pm (-e_2)=\pm e_2
\label{e1e3e2}
\end{equation}
and
\begin{equation}
e_2\t e_3=\pm e_2\t(e_1\t e_2)=\pm e_2\t(e_2\t e_1)=\pm e_1.
\label{e2e3e1}
\end{equation}
This shows that, in this case, $w_1\t w_2\in W$ for any $w_1,w_2\in W$. 

For Type-I, recall that
\begin{equation}
W_{I}^{\perp, 2}=\left\{v\in V:\langle w_1\t w_2, v\rangle_{\vp}=0\text{ for all }w_1,w_2\in W\right\}
\end{equation}
showing that, in this case, $W_I^{\perp, 2}=W^{\perp}$ and hence $W\cap W_I^{\perp,2}=\{0\}$; therefore, in the case that $(W,\t)\cong(\Im\mathbb{H},\t_0)$, we get that $W$ is $1$-multisymplectic using Type-I orthogonality. 

For Type-II, recall that
\begin{equation}
\begin{split}
W_{II}^{\perp,2}&=\{v_1\w v_2\in G^2V:(v_1\w v_2\w w)\lrcorner\vp=0\text{ for all }w\in W\}\\
&=\{v_1\w v_2:\langle v_1\t v_2, w\rangle_{\vp}=0\text{ for all }w\in W\}\\
\end{split}
\end{equation}
By the above, we have that $w_1\t w_2\in W$ for all $w_1,w_2\in W$; therefore, $w_1\w w_2\not\in W_{II}^{\perp,2}$ since $\langle w_1\t w_2, w_1\t w_2\rangle_{\vp}\neq 0$. Hence $G^2W\cap W_{II}^{\perp,2}=\{0\}$, and thus, in the case that $(W,\t)\cong(\Im\mathbb{H},\t_0)$, $W$ is $2$-multisymplectic by Type-II orthogonality. 

Next, let $\{e_1,e_2,e_3\}$ again be an orthonormal basis for $W$, but now assume that $e_i\t e_j\not\in W$ for all $1\leq i, j\leq 3$, $i\neq j$. Then $$\{e_1,e_2,e_3,e_1\t e_2, e_1\t e_3, e_2\t e_3\}$$ is necessarily a linearly independent orthonormal set. Extend this to an orthonormal basis 
\begin{equation}
 \{e_1,e_2,e_3,e_1\t e_2, e_1\t e_3, e_2\t e_3,\tilde{e}\}
\end{equation}
for $V$. We remark that since the identity $x\t(x\t y)=-\left|x\right|^2y+\langle x,y\rangle_{\vp} x$ holds for any $x,y\in V$, Equations \eq{e1e2e3}, \eq{e1e3e2} and \eq{e2e3e1} above must still hold, in which case we can take $\tilde{e}=e_1\t(e_2\t e_3)$.

Note that since $e_i\t e_j\not\in W$, we necessarily have, 
\begin{equation}
 \langle e_i\t e_j, e_k\rangle_{\vp}=0\text{ for all }i,j,k=1,2,3.
\end{equation}
showing both that $W\subseteq W_{I}^{\perp, 2}$ and $G^2W\subseteq W_{II}^{\perp, 2}$. By construction, however, we also have $\tilde{e}\in W_I^{\perp,2}$ and $e_1\w\tilde{e}\in W_{II}^{\perp,2}$ showing that $W$ is Type-I $2$-isotropic and Type-II $2$-isotropic respectively.
\end{proof}

\begin{prop}
For $\dim W=4$,
\begin{itemize}
    \item if $W$ contains a $3$-dimensional subspace $(\tilde{W},\t)$ isomorphic to $(\Im\mathbb{H},\t_0)$, $W$ is Type-I $2$-coisotropic but is none of Type-II $2$-isotropic, Type-II $2$-coisotropic nor $2$-multisymplectic.
    \item otherwise, $W$ is Type-I $2$-Lagrangian and Type-II $2$-isotropic.
\end{itemize}
\end{prop}

\begin{proof}
First assume that $W$ contains a $3$-dimensional subspace $\tilde{W}$ isomorphic to $\Im\mathbb{H}$, and let $\{e_1,e_2,e_3\}$ be an orthonormal basis for $\tilde{W}$ such that $e_1\t e_2=e_3$; extend this to an orthonormal basis $\{e_1,e_2,e_3,f\}$ for $W$. Then for $e_i\t f$ for $1\leq i\leq 3$ and $j\neq i$
\begin{equation}
\begin{split}
&\langle e_i\t f, e_j\rangle_{\vp}=\vp(e_i,f,e_j)=\vp(e_j,e_i,f)\\
=&\langle e_j\t e_i, f\rangle_{\vp}=\langle\pm e_k,f\rangle_{\vp}=0\\
\end{split}
\end{equation}
where $k\neq i,j$. Thus, $\{e_1,e_2,e_3,f,e_1\t f,e_2\t f,e_3\t f\}$ is a basis for $V$.

For Type-I, we note that since $e_1\t e_2=e_3$, $W\neq W_I^{\perp,2}$. In fact, $e_i\not\in W_I^{\perp,2}$ since $span\{e_1,e_2,e_3\}$, together with $\t$, is isomorphic to $(\Im\mathbb{H},\t_0)$; moreover, $e_i\t f\not\in W_I^{\perp,2}$ since 
\begin{equation}
 \langle e_i\t f, e_i\t f\rangle_{\vp}\neq 0
\end{equation}
However, we note that $e_i\t e_j\in span\{e_1,e_2,e_3\}$ implies
\begin{equation}
 \langle e_i\t e_j, f\rangle_{\vp}=0
\end{equation}
and that properties of the cross product yield
\begin{equation}
 \langle e_i\t f, f\rangle_{\vp}=0
\end{equation}
Hence $W_I^{\perp,2}=span\{f\}$ showing that $W_I^{\perp,2}\subsetneq W$ and thus that, in the case that $W$ contains a $3$-dimensional subspace $\tilde{W}$ isomorphic to $\Im\mathbb{H}$, $W$ is Type-I $2$-coisotropic.

For Type-II, since $e_1\t e_2\in W$, $e_1\w e_2\not\in W_{II}^{\perp,2}$ showing that $G^2W\not\subset W_{II}^{\perp,2}$. Further, consider $e_1\t(e_2\t f)$. Note that $e_1\t e_2=e_3$, $e_1\t e_3=-e_2$ and $e_1\t (e_1\t f)=\pm f$; thus, we must have either $e_1\t (e_2\t f)=\pm e_1\t f$ or $e_1\t (e_2\t f)=\pm e_3\t f$, and since $e_1\t (e_2\t f)=\pm e_1\t f$ would imply that $\pm e_2\t f=\pm f$, we see that it must be the case that $e_1\t (e_2\t f)=\pm e_3\t f$. Hence $e_1\w (e_2\t f)\in W_{II}^{\perp,2}$, but clearly $e_1\w (e_2\t f)\not\in G^2W$ showing that $W_{II}^{\perp,2}\not\subset G^2W$. Thus, in the case that $W$ contains a $3$-dimensional subspace $\tilde{W}$ isomorphic to $\Im\mathbb{H}$, $W$ is neither $2$-isotropic nor $2$-coisotropic; moreover, since $e_i\w f\in G^2W\cap W_{II}^{\perp,2}$, we see also that in this case, $W$ is not $2$-multisymplectic. 

Now assume that $W$ does not contain a $3$-dimensional subspace $\tilde{W}$ isomorphic to $\Im\mathbb{H}$. Let $e_1,e_2\in W$ be any nonzero orthogonal unit vectors. Then we must have $e_1\t e_2\in W^{\perp}$ since otherwise $\{e_1,e_2,e_1\t e_2\}$ would be a $3$-dimensional subspace of $W$ isomorphic to $\Im\mathbb{H}$. Note that $\dim\left(\left(span\left\{e_1,e_2,e_1\t e_2\right\}\right)^{\perp}\right)=4$, so
\begin{equation}
\dim\left(\left(span\left\{e_1,e_2,e_1\t e_2\right\}\right)^{\perp}\cap W\right)\geq 1
\end{equation}
Thus, let $f\in \left(span\left\{e_1,e_2,e_1\t e_2\right\}\right)^{\perp}\cap W$ be a unit vector, so $\{e_1,e_2,e_1\t e_2,f\}$ is an orthonormal set such that $\{e_1,e_2,f\}$ spans a $3$-dimensional subspace of $W$ that is not isomorphic to $\Im\mathbb{H}$. Consider $e_1\t f$ and $e_2\t f$. Again, since $W$ cannot contain a $3$-dimensional subspace isomorphic to $\Im\mathbb{H}$, we must have $e_1\t f, e_2\t f\in W^{\perp}$, so that $\{e_1\t e_2, e_1\t f, e_2\t f\}$ is a basis for $W^{\perp}$ which, together with $\t$, is necessarily isomorphic to $\Im\mathbb{H}$. Let $\tilde{f}\in W$ such that $\{e_1,e_2,f,\tilde{f}\}$ forms an orthonormal basis for $W$. It will be useful to note that also $e_1\t\tilde{f}, e_2\t\tilde{f}, f\t\tilde{f}\not\in W$, so necessarily, $e_1\t\tilde{f}, e_2\t\tilde{f}, f\t\tilde{f}\in W^{\perp}$.

For Type-I, this yields
\begin{equation}
\begin{split}
W_I^{\perp,2}&=\{v\in V:\langle w_1\t w_2, v\rangle_{\vp}=0\text{ for all }w_1,w_2\in W\}\\
&=\left(W^{\perp}\right)^{\perp}\\
&=W
\end{split}
\end{equation}
showing that a $4$-dimensional subspace $W$ that does not contain a $3$-dimensional subspace isomorphic to $\Im\mathbb{H}$ is Type-I $2$-Lagrangian.

For Type-II, the same considerations show that $G^2W\subseteq W_{II}^{\perp,2}$, and since $(e_1\t e_2)\t(e_1\t f)=\pm e_2\t f$, we necessarily have $(e_1\t e_2)\w (e_1\t f)\in W_{II}^{\perp, 2}-G^2W$ showing that a $4$-dimensional subspace $W$ that does not contain a $3$-dimensional subspace isomorphic to $\Im\mathbb{H}$ is Type-II $2$-isotropic
\end{proof}

\begin{prop}
For $\dim W=5$, $W$ is Type-I $2$-coisotropic and $1$-multisymplectic, but is none of Type-II $2$-isotropic, Type-II $2$-coisotropic nor Type-II $2$-multisymplectic.
\end{prop}

\begin{proof}
Let $\dim W=5$. Assume first for contradiction that $W$ does not contain a $3$-dimensional subspace isomorphic to $\Im\mathbb{H}$, and let $e_1,e_2\in W$ be orthogonal unit vectors. Then, we must have $e_1\t e_2\in W^{\perp}$, and since
\begin{equation*}
\dim\left(\left(span\left(e_1,e_2,e_1\t e_2\right)\right)^{\perp}\cap W\right)\geq 2
\end{equation*}
we can pick $f\in \left(\left\{e_1,e_2,e_1\t e_2\right\}\right)^{\perp}\cap W$ to be a unit vector, so $\{e_1,e_2,e_1\t e_2,f\}$ is an orthonormal set such that $\{e_1,e_2,f\}$ is not a $3$-dimensional subspace isomorphic to $\Im\mathbb{H}$. Consider $e_1\t f$ and $e_2\t f$. Again, since $W$ cannot contain a $3$-dimensional subspace isomorphic to $\Im\mathbb{H}$, we must have $e_1\t f, e_2\t f\in W^{\perp}$. Consider the subset $\{e_1\t e_2, e_1\t f, e_2\t f\}$ of $W^{\perp}$. Of course, since $\dim\left(W^{\perp}\right)=2$, this set must be linearly dependent; however, because $\{e_1\t e_2, e_1\t f, e_2\t f\}$ is an orthonormal set, it must be linearly independent, a contradiction. Hence, $W$ must contain a $3$-dimensional subspace isomorphic to $\Im\mathbb{H}$.

Let $\{e_1,e_2,e_3\}$ now denote an orthonormal basis for a $3$-dimensional subspace of $W$ isomorphic to $\Im\mathbb{H}$ satisfying, without loss of generality, $e_1\t e_2=e_3$. Consider the orthogonal complement of $span\left\{e_1,e_2,e_3\right\}$ in $W$ with respect to the metric $\langle\cdot,\cdot\rangle_{\vp}$, and let $f$ be any nonzero vector in this space. Thus, $\left\{e_1,e_2,e_3,f\right\}$ is an orthonormal set and hence is linearly independent; moreover, we must have $e_i\t f\in W$ for some $i=1,2,3$ since otherwise, because $\{e_1\t f, e_2\t f, e_3\t f\}$ is a linearly independent set, we would have
\begin{equation}
\dim\left(W+span\left(e_1\t f, e_2\t f, e_3\t f\right)\right)=8>\dim V.
\end{equation}
Then $\{e_1,e_2,e_3,f,e_i\t f\}$ is a basis for $W$ and $\{e_1,e_2,e_3,f,e_1\t f,e_2\t f,e_3\t f\}$ is a basis for $V$. Finally, we note that
\begin{equation}
 \left\{w_1\t w_2:w_1,w_2\in W\right\}=V
\end{equation}
To see this, note first that we must have $\{e_1\t f,e_2\t f,e_3\t f\}\subset\left\{w_1\t w_2\right\}$ by definition; $\{e_1,e_2,e_3\}\subset\left\{w_1\t w_2\right\}$ by construction of the set $\{e_1,e_2,e_3\}$. Finally, $f=\pm e_i\t(e_i\t f)\in\left\{w_1\t w_2\right\}$. 

For Type-I, we then have
\begin{equation}
\begin{split}
W_I^{\perp,2}&=\{v\in V:\langle w_1\t w_2, v\rangle_{\vp}=0\text{ for all }w_1,w_2\in W\}\\
&=V^{\perp}\\
&=\{0\}
\end{split}
\end{equation}
showing that, for $\dim W=5$, $W$ is Type-I $2$-coisotropic and $1$-multisymplectic.

For Type-II, let $e_j$ such that $e_j\t f\not\in W$. Then $e_i\t(e_j\t f)=e_k\t f$ where $k\neq j$ is such that $e_k\t f\not\in W$ showing that $e_i\w(e_j\t f)\in W_{II}^{\perp,2}$ and hence that $W_{II}^{\perp,2}\not\subset G^2W$; moreover, $e_1\w e_2\not\in W_{II}^{\perp,2}$ since $e_1\t e_2\in W$ showing that $G^2W\not\subset W_{II}^{\perp,2}$. However, $e_j\w f\in G^2W\cap W_{II}^{\perp, 2}$. Thus for $\dim{W}=5$, $W$ is none of Type-II $2$-isotropic, Type-II $2$-coisotropic nor $2$-multisymplectic.
\end{proof}

\begin{prop}
For $\dim W=6$, $W$ is Type-I and Type-II $2$-coisotropic and $1$-multisymplectic.
\end{prop}

\begin{proof}
As above, $W$ must contain a $3$-dimensional subspace isomorphic to $\Im\mathbb{H}$. Let $\{e_1,e_2,e_3\}$ now denote an orthonormal basis for a $3$-dimensional subspace of $W$ isomorphic to $\Im\mathbb{H}$ satisfying, without loss of generality, $e_1\t e_2=e_3$. Consider the orthogonal complement of $span\left\{e_1,e_2,e_3\right\}$ in $W$ with respect to the metric $\langle\cdot,\cdot\rangle_{\vp}$, and let $f$ be any nonzero vector in this space. Thus, $\left\{e_1,e_2,e_3,f\right\}$ is an orthonormal set and hence is linearly independent; moreover, we must have $e_i\t f,e_j\t f\in W$ for some $1\leq i<j\leq3$ since otherwise, because $\{e_1\t f, e_2\t f, e_3\t f\}$ is a linearly independent set, we would have
\begin{equation}
\dim\left(W+span\left\{e_1\t f, e_2\t f, e_3\t f\right\}\right)\geq8>\dim V.
\end{equation}
Then $\{e_1,e_2,e_3,f,e_i\t f, e_j\t f\}$ is a basis for $W$, $\{e_1,e_2,e_3,f,e_1\t f,e_2\t f,e_3\t f\}$ is a basis for $V$. Again, we necessarily have $\{w_1\t w_2:w_1,w_2\in W\}=V$.

For Type-I, we then have
\begin{equation}
\begin{split}
W_I^{\perp,2}&=\{v\in V:\langle w_1\t w_2, v\rangle_{\vp}=0\text{ for all }w_1,w_2\in W\}\\
&=V^{\perp}\\
&=\{0\}
\end{split}
\end{equation}
showing that $W_I^{\perp,2}\subsetneq W$. Hence for $\dim W=6$, $W$ is Type-I $2$-coisotropic and $1$-multisymplectic.

For Type-II, let $v_1\w v_2\in W^{\perp,2}_{II}$. Then, since $\langle v_1\t v_2,w\rangle_{\vp}=0$ for all $w\in W$, we necessarily have $v_1\t v_2\in span\{e_k\t f\}$ where $e_k\t f\not\in W$. Using properties of the cross product, this implies that 
\begin{equation}
v_1\w v_2\in span\left\{e_i\w(e_j\t f), e_j\w(e_i\t f), e_k\w f\right\}
\end{equation}
showing that $W_{II}^{\perp,2}\subsetneq G^2W$. Hence for $\dim W=6$, $W$ is Type-II $2$-coisotropic.
\end{proof}

\section{Associative \& Coassociative Submanifolds}
In $G_2$-geometry, there are two particular classes of submanifolds that have been studied by many authors, namely, associative and coassociative submanifolds. Let $(M,\vp)$ be a $G_2$-manifold so that $\d\vp=0=\d^*\vp$. Then both $\vp$ and $\star\vp$ are \emph{calibrations} in the sense of Harvey and Lawson \cite{HaLa}. The \emph{associative} submanifolds are then those calibrated by $\vp$, and the \emph{coassociative} submanifolds are those calibrated by $\star\vp$; stated differently, the associative submanifolds are those $3$-dimensional submanifolds for which $\vp$ restricts to be the induced volume form, and the coassociative submanifolds are those $4$-dimensional submanifolds for which $\star\vp$ restricts to be the induced volume form. In \cite{HaLa}, it is shown that a $4$-dimensional submanifold is coassociative if and only if $\vp$ restricts to be identically zero on the $4$-manifold.

\begin{dfn}
A $3$-dimensional subspace $A$ of a $G_2$-vector space $(V,\vp)$ will be called an \emph{associative} subspace if $\vp$ restricts to be the induced volume form on $A$. A $4$-dimensional subspace $C$ of a $G_2$-vector space $(V,\vp)$ will be called a \emph{coassociative} subspace if $\vp$ restricts to be zero on $C$.
\end{dfn}

\begin{cor}
Let $(V,\vp)$ be a $G_2$ vector space. Then the associative subspaces are the $3$-dimensional subspaces which are $2$-multisymplectic and the coassociative subspaces are the $4$-dimensional subspaces which are Type-I $2$-Lagrangian/Type-II $2$-isotropic.
\end{cor}

\begin{proof}
Let $A$ be an associative subspace. Since $\vp$ restricts to be the volume form on $A$, $A$ is neither Type-I nor Type-II $2$-isotropic. Hence, $A$ must be $2$-multisymplectic. Conversely, let $A$ be any $3$-dimensional, $2$-multisymplectic subspace of $(V,\vp)$. $A$ is $2$-multisymplectic means that $\vp|_A$ is fully nondegenerate on $A$, showing that $\vp|_A$ must be the volume form on $A$. Hence $A$ is an associative subspace.

Let $C$ be a coassociative subspace. Now, that $\vp$ restricts to zero on $C$ means that for any $c_i,c_j,c\in C$
\begin{equation}
0=\vp(c_i,c_j,c)=\langle c_i\t c_j,c\rangle_{\vp}
\end{equation}
This implies that $C$ is Type-I/Type-II $2$-isotropic, from which we further get that $C$ must be Type-I $2$-Lagrangian. Conversely, if $W$ is any $4$-dimensional subspace which is Type-I $2$-Lagrangian, then, by definition, we have $W=W_I^{\perp,2}$. Thus, for any $w,w_1,w_2\in W$, we must have
\begin{equation}
(w\w w_1\w w_2)\lrcorner\vp=0
\end{equation}
since $w\in W=W_I^{\perp,2}$. Thus, $W$ is a coassociative subspace. Similarly, if $W$ is any $4$-dimensional subspace which is Type-II $2$-isotropic, then, by definition, we have $G^2W\subset W_{II}^{\perp,2}$. Thus, for any $w_1,w_2,w\in W$, we must have
\begin{equation}
(w_1\w w_2\w w)\lrcorner\vp=0
\end{equation}
since $w_1\w w_2\in G^2W\subset W_{II}^{\perp,2}$.
\end{proof}

The following corollary is now immediate from the fact that, given a $G_2$ manifold, the tangent space at each point of an associative (resp. coassociative) submanifold is an associative (resp. coassociative) subspace of the $G_2$ vector space $(T_xM,\vp_x)$.

\begin{cor}
Let $(M,\vp)$ be a $G_2$ manifold. Then the associative submanifolds are exactly those $3$-dimensional submanifolds which are $2$-multisymplectic, and the coassociative submanifolds are exactly those $4$-dimensional submanifolds which are Type-I $2$-Lagrangian/Type-II $2$-isotropic.
\end{cor}

%%%%%%%%%%%%%%%%%%%%%%%%%%%%%%%%%%%%%%%%%%%%%%%%%%%%%%%%%%%%%%%%%%%%%%%%%%%%%
%%%%%%%%%%%%%%%%%%%%%%%%%%%%%%%%%%%%%%%%%%%%%%%%%%%%%%%%%%%%%%%%%%%%%%%%%%%%%
\bibliographystyle{amsplain}
\bibliography{ajbibliography}
%%%%%%%%%%%%%%%%%%%%%%%%%%%%%%%%%%%%%%%%%%%%%%%%%%%%%%%%%%%%%%%%%%%%%%%%%%%%%
%%%%%%%%%%%%%%%%%%%%%%%%%%%%%%%%%%%%%%%%%%%%%%%%%%%%%%%%%%%%%%%%%%%%%%%%%%%%%
\end{document}